\documentclass[11pt]{amsart}
\usepackage{graphicx,subcaption}
\usepackage{amsmath} 
\usepackage{amsthm,amsfonts,amssymb,mathrsfs,amscd,amstext,amsbsy} 
\usepackage{epic,eepic} 
\usepackage{yfonts}
\usepackage{paralist,enumerate}
\usepackage[all]{xy}
\usepackage{tikz}
\usepackage{hyperref}
\hypersetup{colorlinks}
\usetikzlibrary {positioning}
\definecolor {processblue}{cmyk}{0.96,0,0,0}
\usetikzlibrary{arrows}
\tikzset{    vertex/.style={circle,draw,minimum size=1.5em},    edge/.style={->,> = latex'}}

\newtheorem{theorem}{Theorem}[section]

\newtheorem{conj}[theorem]{Conjecture}

\newtheorem{theo}[theorem]{Theorem}
\newtheorem{lem}[theorem]{Lemma}
\newtheorem{pro}[theorem]{Proposition}
\newtheorem{rem}[theorem]{Remark}
\newtheorem{exa}[theorem]{Example}
\newtheorem{que}[theorem]{Question}

\newtheorem{Definition}[theorem]{Definition}
\newtheorem*{Definition*}{Definition}

\def\qed{\hfill \ifhmode\unskip\nobreak\fi\quad\ifmmode\Box\else$\Box$\fi\\ }

\begin{document}

\title[Oriented $S^1$-manifold with 3 fixed points]{Circle actions on oriented manifolds with 3 fixed points}
\author{Donghoon Jang}
\address{Department of Mathematics, Pusan National University, Pusan 46241, Korea}
\email{donghoonjang@pusan.ac.kr}

\begin{abstract}
Let the circle group act on a compact oriented manifold $M$ with a non-empty discrete fixed point set. Then the dimension of $M$ is even. If $M$ has one fixed point, $M$ is the point. In any even dimension, such a manifold $M$ with two fixed points exists, a rotation of an even dimensional sphere. Suppose that $M$ has three fixed points. Then the dimension of $M$ is a multiple of 4. Under the assumption that each isotropy submanifold is orientable, we show that if $\dim M=8$, then the weights at the fixed points agree with those of an action on the quaternionic projective space $\mathbb{HP}^2$, and show that there is no such 12-dimensional manifold $M$.
\end{abstract}

\maketitle
\section{Introduction}

The study of rotations of geometric objects, such as rotations of spheres, is a classical topic in geometry and topology. A rotation of an even-dimensional sphere $S^{2n}$ has two fixed points, and there are many examples of circle actions on compact manifolds that have finite fixed point sets. If the circle group acts on a manifold with a non-empty finite fixed point set, the dimension of the manifold is necessarily even. This is because the tangent space of each isolated fixed point decomposes into real $2$-dimensional $S^1$-representations; also see the discussion after Theorem \ref{t11}. Therefore, it is a natural question to ask, given $2n$ and $k$, whether there exists a $2n$-dimensional $S^1$-manifold $M$ with $k$ fixed points. 

If we consider non-orientable manifolds, the answer is known; any pair $(2n,k)$ is possible, and we can see this as follows. The real projective space $\mathbb{RP}^{2n}$ admits an action with one fixed point, as the quotient of a rotation on $S^{2n}$ with 2 fixed points under the antipodal map. Thus, for any pair $(2n,k)$, we can produce a $2n$-dimensional compact $S^1$-manifold with $k$ fixed points, for instance by taking an equivariant connected sum along free orbits of $k$-copies of the action on $\mathbb{RP}^{2n}$. On the other hand, for oriented manifolds, a complete answer is not known, and it is the focus of this paper. By choosing an orientation, we may assume that any orientable manifold is oriented, and results of this paper apply to orientable manifolds; for simplicity, to avoid choosing an orientation all the time, we will always assume that any orientable manifold is oriented. Throughout this paper, any manifold is assumed to be smooth and have no boundary, and any group action on a manifold is assumed to be smooth.

\begin{que} \label{q11} Suppose that $M$ is a $2n$-dimensional compact oriented manifold and the circle group $S^1$ acts on $M$ with $k$ fixed points, $k \in \mathbb{N}$. Then what are the possible pairs $(2n,k)$?
\end{que}

If we approach Question \ref{q11} in terms of the number $k$ of fixed points, the following are well-known for a pair $(2n,k)$.
\begin{enumerate}
\item If $k=1$, then $2n=0$. That is, if there is exactly one fixed point, $M$ must be the point itself. For this fact, for instance, see Lemma \ref{l23}.
\item If $k=2$, any $2n$ is possible, and a rotation of an even-dimensional sphere $S^{2n}$ provides an example in any even dimension.
\item More generally, if $k$ is even, any $2n$ is possible; a disjoint union of copies of rotations of $S^{2n}$ is an example. To get a connected manifold, take a connected sum along free orbits of rotations of $S^{2n}$ for $n>1$; if $n=1$, $k$ must be 2 and it is the 2-sphere.
\item If $k$ is odd, then $2n$ must be a multiple of 4; see Lemma \ref{c27}.
\item If $k=3$, then there are examples for $2n=4$, $8$, and $16$; the complex projective $\mathbb{CP}^2$, the quaternionic projective space $\mathbb{HP}^2$, and the octonionic projective space $\mathbb{OP}^2$ each admit a circle action with 3 fixed points. Their real dimensions are 4, 8, and 16, respectively.
\end{enumerate}

For the actions on $\mathbb{HP}^2$ and $\mathbb{OP}^2$, see Example \ref{e214} and \cite{GW}, respectively.

We approach Question \ref{q11} from low dimensions. If $2n \leq 10$, the answer to Question \ref{q11} is also well-known and is as in Table \ref{table}.

\begin{table}[]
\begin{tabular}{|l|l|l|}
\hline
2n & Possible $k$   & Examples                                \\ \hline
0  & Any $k \geq 1$ & $k$ points                             \\ \hline
2  & Any even $k$   & Copies of $S^2$                        \\ \hline
4  & Any $k \geq 2$ & Copies of $S^4$ and $\mathbb{CP}^2$ \\ \hline
6  & Any even $k$   & Copies of $S^6$                        \\ \hline
8  & Any $k \geq 2$ & Copies of $S^8$ and $\mathbb{HP}^2$ \\ \hline
10 & Any even $k$   & Copies of $S^{10}$                     \\ \hline
\end{tabular}
\caption{Possible $(2n,k)$ in low dimensions} \label{table}
\end{table}

Therefore, both from small numbers of fixed points and from low dimensions, the first case for which an answer is not known is $(2n,k)=(12,3)$.

\begin{que}\label{q12}
Does there exist a circle action on a 12-dimensional compact oriented manifold with exactly 3 fixed points?
\end{que}

The purpose of this paper is to answer Question \ref{q12} under an assumption.

\begin{theo} \label{t11}
There does not exist a circle action on a 12-dimensional compact oriented manifold with exactly 3 fixed points, provided that the isotropy submanifolds are orientable.
\end{theo}

\begin{rem}
Let the circle act on an orientable manifold $M$. Let $w$ be a positive integer. Consider a cyclic subgroup $\mathbb{Z}_w$ of $S^1$ and its action on $M$. If $w$ is odd, then the isotropy submanifold $M^{\mathbb{Z}_w}$ is orientable. Therefore, the assumption in Theorem \ref{t11} is that the isotropy submanifold $M^{\mathbb{Z}_w}$ is orientable for each even positive integer $w$. In fact, we only need that any connected component of $M^{\mathbb{Z}_w}$ containing an $S^1$-fixed point is orientable for each even positive integer $w$; with this assumption the proof of Theorem \ref{t11} holds.
\end{rem}

Consider a circle action on a compact oriented manifold $M$. Suppose that $p$ is an isolated fixed point. The tangent space $T_pM$ to $M$ at $p$ has a decomposition
\begin{center}
$T_pM=L_{p,1} \oplus \cdots \oplus L_{p,n}$ for some $n$,
\end{center}
where each $L_{p,i}$ is a real 2-dimensional irreducible $S^1$-equivariant real vector space. Such a vector space $L_{p,i}$ is isomorphic to a complex 1-dimensional $S^1$-equivariant complex space on which $S^1$ acts by multiplication by $g^{w_{p,i}}$ for some non-zero integer $w_{p,i}$ for all $g \in S^1$. In particular, the dimension of $M$ is $2n$ and even. For each $L_{p,i}$, we choose an orientation of $L_{p,i}$ so that $w_{p,i}$ is positive, and call the positive integers $w_{p,i}$ the \textbf{weights} at $p$. The tangent space $T_pM$ at $p$ has two orientations; one coming from the orientation on $M$ and the other coming from the orientation on the representation space $L_{p,1} \oplus \cdots \oplus L_{p,n}$. Let $\epsilon(p)=+1$ if the two orientations agree and $\epsilon(p)=-1$ otherwise, and call it the \textbf{sign} of $p$. Let $\{w_{p,1},\cdots,w_{p,n}\}$ be the unordered multiset of the weights at $p$. We define the \textbf{fixed point data} of $p$ by $\Sigma_p:=(\epsilon(p) ,\{w_{p,1}, \cdots, w_{p,n}\})$. Inside $\Sigma_p$, for the sign of $p$ we will omit 1 and only write its sign, not to be confused with weights. To simplify our notation, for the fixed point data of $p$, we shall use the notation $\{\epsilon(p),w_{p,1},\cdots,w_{p,n}\}$ instead of $(\epsilon(p) ,\{w_{p,1}, \cdots, w_{p,n}\})$. We define the \textbf{fixed point data} of $M$ to be the multiset of fixed point datum of the fixed points and denote it by $\Sigma_M$, that is, $\Sigma_M=\cup_{p \in M^{S^1}} \{\Sigma_p\}$. The study of these numerical invariants of fixed points is a key tool in the proof of Theorem \ref{t11}.

As mentioned above, if $M$ has exactly 1 fixed point $p$, $M$ has to be the fixed point itself, that is, $M=\{p\}$. If $M$ has exactly two fixed points, then the fixed point data of $M$ agrees with that of a rotation of $S^{2n}$ \cite{K3}, see Theorem \ref{t213}. 

Next, suppose that $M$ has exactly three fixed points. Then $\dim M$ is a multiple of 4 as mentioned above. If $\dim M=4$, the fixed point data of $M$ agrees with that of an action on $\mathbb{CP}^2$, see \cite{J3}. 
If $\dim M=8$ and each isotropy submanifold is orientable, the fixed point data of $M$ is the same as the data of the action on $\mathbb{HP}^2$.

\begin{theo} \label{t12}
Let the circle act on an 8-dimensional compact oriented manifold $M$ with 3 fixed points. Suppose that each isotropy submanifold is orientable. Then the fixed point data of $M$ is
\begin{center}
$\{\pm, a+b, a+c, a, a+b+c\}$, $\{\pm, a+b, a+c, b, c\}$, and $\{\mp, a, a+b+c, b, c\}$
\end{center}
for some positive integers $a$, $b$, and $c$.
\end{theo}

In Theorem \ref{t12}, the signs mean that first two fixed points have the same sign and the other has the opposite sign. As noted earlier, Example \ref{e214} illustrates an action on $\mathbb{HP}^2$ with the fixed point data described above.

In \cite{Ku}, Kustarev obtained a result similar to Theorem \ref{t12}. The result of the paper says that two types of fixed point data occur for an 8-dimensional compact oriented $S^1$-manifold with 3 fixed points; in \cite{Ku} it is also assumed that each isotropy submanifold is orientable. On the other hand, as one sees from Theorem \ref{t12} and Example \ref{e214}, only one type of fixed point data occurs, and the two types of fixed point data in \cite{Ku} are in fact equivalent. Our proof is simple; see our proof in Section \ref{s5}.

We discuss recent results on circle actions on orientable manifolds with 3 fixed points. In \cite{W}, Wiemeler showed that if the circle group acts on a closed orientable manifold with 3 fixed points, then the dimension of the manifold is of the form $4 \cdot 2^a$ or $8 \cdot (2^a+2^b)$ with $a, b \geq 0$ and $a \neq b$. For this, Wiemeler utilizes the equivariant cohomology, the Hirzebruch signature theorem, and the Pontryagin classes and numbers, and formulates some arguments on the signature and the Pontryagin classes in \cite{FS, S} for rational projective planes. In \cite{DW}, Dong and Wang improved this result by showing that if the circle group acts on a closed orientable manifold $M$ with 3 fixed points, then $\dim M \in \{4, 8, 16\}$.

\subsection{Idea of proof of Theorem \ref{t11}}

The idea of the proof of Theorem \ref{t11} is to study the fixed point data of an $S^1$-manifold. To elucidate, let the circle act on a compact oriented manifold $M$ with a discrete fixed point set such that its isotropy submanifolds are orientable. There exists a multigraph that contains information on the fixed point data of $M$; a vertex corresponds to a fixed point, and the label of an edge is a weight of its vertex (see Definitions \ref{d28} and \ref{d29} and Proposition \ref{p210}). In particular, if $M$ has exactly 3 fixed points, the multigraph says that the fixed point data of $M$ has a particular form (Proposition \ref{p41}). Proposition \ref{p42} and Lemmas \ref{l61}, \ref{l62}, and \ref{l66} provide relationships that the weights at the fixed points must satisfy for $M$ with 3 fixed points. We prove Proposition \ref{p42} by applying the Atiyah-Singer index theorem (Theorem \ref{t22}) to the setting in Proposition \ref{p41}. We use the equivariant Pontryagin classes in the Atiyah-Bott-Berline-Vergne (ABBV) localization theorem (Theorem \ref{t21}) to obtain Lemma \ref{l34}, which yields Lemma \ref{l61}. Lemma \ref{l66} is obtained from the isomorphism of the weights at the fixed points lying in an isotropy submanifold. We prove Theorem \ref{t11} by showing that if $\dim M=12$ and $M$ has exactly 3 fixed points, the fixed point of $M$ cannot satisfy all of the aforementioned relationships between the weights at the fixed points, implying that such an $M$ does not exist.

\subsection{Organization}

This paper is organized as follows.
Let $M$ be a compact oriented manifold equipped with a circle action having a non-empty discrete fixed point set. In Section \ref{s2}, we review background. In Section \ref{multigraph}, we discuss a multigraph encoding the fixed point data of $M$. In Section \ref{s3}, we prove a few properties on Pontryagin classes and Pontryagin numbers. Suppose that $M$ has exactly 3 fixed points and each isotropy submanifold is orientable. In Section \ref{s4}, we prove preliminary results that the fixed point data of $M$ satisfy. In Section \ref{s5}, we prove Theorem \ref{t12} that if $\dim M=8$, the fixed point data of $M$ agrees with that of an action on $\mathbb{HP}^2$ in Example \ref{e214}. In Section \ref{s6}, we further investigate properties that the fixed point data of $M$ satisfies if $\dim M=12$. In Section \ref{s7} we prove Theorem \ref{t11} that there does not exist a 12-dimensional compact oriented $S^1$-manifold with 3 fixed points. In Section \ref{s8}, we discuss a relation between the dimension of a compact oriented $S^1$-manifold and the number of fixed points when it is odd, and compare with results for other types of $S^1$-manifolds.

\section{Background} \label{s2}

In this section, we review necessary background.
We review the ABBV localization theorem (Theorem \ref{t21}) and the Atiyah-Singer index theorem (Theorem \ref{t22}) and their consequences. In Section \ref{s3}, we use Theorem \ref{t21} to prove relations between the first equivariant Pontryagin classes at the fixed points (Lemma \ref{l34}). In Section \ref{s4}, we apply Theorem \ref{t22} to obtain relations between the weights at the fixed points (Proposition \ref{p42}).

Let a Lie group $G$ act on a manifold $M$. The fixed point set, denoted by $M^G$, is the set of points in $M$ that are fixed by all elements in $G$, that is, $M^G=\{m \in M \, | \, g \cdot m=m, \forall g \in G\}$.

Let the circle act on a manifold $M$. Let $w$ be a positive integer. The cyclic group of order $w$, which we will denote $\mathbb{Z}_w$, is a subgroup of $S^1$ and therefore acts on $M$. The set $M^{\mathbb{Z}_w}$ of points in $M$ that are fixed by the $\mathbb{Z}_w$-action is a union of smaller dimensional compact submanifolds of $M$. These smaller dimensional
submanifolds are called \textbf{isotropy submanifolds}.

Let the circle act on a manifold $M$. The \textbf{equivariant cohomology} of $M$ is $\displaystyle H_{S^{1}}^{*}(M)=H^{*}(M \times_{S^{1}} S^{\infty})$. Suppose that $M$ is oriented and compact. The projection map $\pi : M \times_{S^{1}} S^{\infty} \rightarrow \mathbb{CP}^{\infty}$ induces a natural push-forward map
\begin{center}
$\displaystyle \pi_{*} : H_{S^{1}}^{i}(M;\mathbb{Z}) \rightarrow H^{i-\dim M}(\mathbb{CP}^{\infty};\mathbb{Z})$
\end{center}
for $i \in \mathbb{Z}$. This map is often denoted by $\int_{M}$ and is given by integration over the fiber.

\begin{theo} \emph{[ABBV localization theorem]} \cite{AB} \label{t21}
Let the circle act on a compact oriented manifold $M$. Let $ \alpha \in H_{S^{1}}^{\ast}(M;\mathbb{Q})$. As an element of $\mathbb{Q}(t)$,
\begin{center}
$\displaystyle \int_{M} \alpha = \sum_{F \subset M^{S^{1}}} \int_{F} \frac{\alpha|_{F}}{e_{S^{1}}(N_{F})}$,
\end{center}
where the sum is taken over all fixed components, and $e_{S^{1}}(N_{F})$ is the equivariant Euler class of the normal bundle to $F$.
\end{theo}

Let $M$ be a compact oriented manifold. The \textbf{$L$-genus} is the genus beloging to the power series $\frac{\sqrt{z}}{\tanh \sqrt{z}}$. The \textbf{signature} of $M$, denoted by $\textrm{sign}(M)$, is the analytical index of the signature operator on $M$. The Atiyah-Singer index theorem states that the analytical index of the signature operator on $M$ is equal to the topological index of the operator, and is equal to the $L$-genus of $M$. Let the circle group acts on $M$. Then we define the equivariant index of the operator \cite{AS}. The equivariant index of the operator is rigid under the action, and is equal to the signature of $M$. If the fixed point set is discrete, the Atiyah-Singer index theorem implies the following formula; recall that for each isolated fixed point $p$, $\epsilon(p)$ is the sign of $p$ and $w_{p,1},\cdots,w_{p,n}$ are the weights at $p$ (see the definitions in the introduction).

\begin{theo} \emph{[Atiyah-Singer index theorem]} \cite{AS} \label{t22}
Let the circle act on a $2n$-dimensional compact oriented manifold $M$ with a discrete fixed point set. The signature of $M$ is
\begin{center}
$\displaystyle{\textrm{sign}(M) = \sum_{p \in M^{S^1}} \epsilon(p) \prod_{i=1}^{n} \frac{1+t^{w_{p,i}}}{1-t^{w_{p,i}}}}$,
\end{center}
for all indeterminates $t$.
\end{theo}

From Theorem \ref{t22} we obtain simple bounds on the signature of $M$; this simple result was proved in \cite{J4} and we shall review the proof.

\begin{lem} \label{l23} \cite{J4}
Let the circle act on a compact oriented manifold $M$ with $k$ fixed points, where $\dim M>0$ and $0<k<\infty$. Then $k \geq 2$ and the signature of $M$ satisfies $|\textrm{sign}(M)| \leq k-2$.
\end{lem}

\begin{proof}
Let $\dim M=2n>0$. Suppose that $\epsilon(p)=+1$ for all $p \in M^{S^1}$. By Theorem \ref{t22}, 
\begin{center}
$\displaystyle \textrm{sign}(M) = \sum_{p \in M^{S^1}} \epsilon(p) \prod_{i=1}^{n} \frac{1+t^{w_{p,i}}}{1-t^{w_{p,i}}}=\sum_{p \in M^{S^1}} \prod_{i=1}^n [(1+t^{w_{p,i}}) (\sum_{j=0}^{\infty} t^{j w_{p,i}})]=\sum_{p \in M^{S^1}} \prod_{i=1}^n(1+2\sum_{j=1}^{\infty}t^{jw_{p,i}})$
\end{center}
for all indeterminate $t$, which cannot hold since the signature of $M$ is a constant. Therefore there exists at least one fixed point $q$ with $\epsilon(q)=-1$. Similarly, there exists at least one fixed point $q'$ with $\epsilon(q')=+1$. Hence $k \geq 2$. Taking $t=0$ in Theorem \ref{t22}, $\textrm{sign}(M)=\sum_{p \in M^{S^1}} \epsilon(p)$. \end{proof}

Manipulating the the Atiyah-Singer index formula (Theorem \ref{t22}) and comparing the coefficients of $t^w$-terms where $w$ is the smallest weight, the following lemma holds.

\begin{lem} \label{l25} \cite{J3} Let the circle act on a $2n$-dimensional compact oriented manifold $M$ with a discrete fixed point set. Let $w=\min\{w_{p,i} \, | \, 1 \leq i \leq n, p \in M^{S^1}\}$. Then
\begin{center}
$\displaystyle \sum_{p \in M^{S^1}, \epsilon(p)=+1} N_p(w)=\sum_{p \in M^{S^1}, \epsilon(p)=-1} N_p(w)$,
\end{center}
where $N_p(w)=|\{i \, : \, w_{p,i}=w, 1 \leq i \leq n, p \in M^{S^1}\}|$ is the number of times $w$ occurs as a weight at $p$. \end{lem}

For a circle action on a compact manifold, the Euler number of the manifold is equal to the sum of Euler numbers of its fixed components.

\begin{theo} \label{t26} \cite{Kob}
Let the circle act on a compact manifold $M$. Then
\begin{center}
$\displaystyle \chi(M)=\sum_{F \subset M^{S^1}} \chi(F)$,
\end{center}
where $\chi(F)$ denote the Euler number of $F$.
\end{theo}

Let $M$ be a compact oriented manifold. Then $\chi(M)=\sum_{i=0}^{\dim M} (-1)^i b_i$, where $b_i$ denotes the $i$-th Betti number of $M$. Let the circle group act on $M$ with a discrete fixed point set. Since the Euler number of a point is 1, Theorem \ref{t26} and Poincar\'e duality imply the following known fact.

\begin{lem} \label{c27} 
Let the circle act on a compact oriented manifold. If the number of fixed points is odd, the dimension of the manifold is divisible by four. \end{lem}

\section{Multigraphs describing oriented $S^1$-manifolds}\label{multigraph}

In this section, we discuss a certain type of multigraph that contains information on the fixed point data of a compact oriented $S^1$-manifold with a discrete fixed point set, under the assumption that each isotropy submanifold is orientable; a vertex of the multigraph corresponds to a fixed point of the manifold, and each edge is labeled by a positive integer, which corresponds to a weight of its fixed point (vertex) (see Definitions \ref{d28} and \ref{d29} and Proposition \ref{p210}). This was established in \cite[Proposition 2.6]{J3}, but it requires the assumption that each isotropy submanifold is orientable. Otherwise, with the assumption, the proof of \cite[Proposition 2.6]{J3} goes smoothly. We shall review the proof in this section. First, we introduce the notion of a signed labeled multigraph.

\begin{Definition} \label{d28} A \textbf{multigraph} $\Gamma$ is an ordered pair $\Gamma=(V,E)$ where $V$ is a set of vertices and $E$ is a multiset of unordered pairs of vertices, called \textbf{edges}. A multigraph is called \textbf{signed} if every vertex $v$ is assigned a number $+1$ or $-1$, called the \textbf{sign} of $v$. A multigraph is called \textbf{labeled}, if every edge $e$ is labeled by a positive integer $w(e)$, called the \textbf{label}, or the \textbf{weight} of the edge. That is, a multigraph is labeled if there exists a map from $E$ to the set of positive integers. Let $\Gamma$ be a labeled multigraph. The \textbf{weights} at a vertex $v$ consist of the labels (weights) $w(e)$ of edges $e$ at $v$. A multigraph $\Gamma$ is called \textbf{$n$-regular}, if every vertex has $n$-edges. \end{Definition}

Next, we introduce a signed labeled multigraph that contains the fixed point data of a compact oriented $S^1$-manifold with a discrete fixed point set.

\begin{Definition} \label{d29} Let the circle act on a compact oriented manifold $M$ with a discrete fixed point set. We say that a (signed labeled) multigraph $\Gamma=(V,E)$ \textbf{describes} $M$ if the following hold.
\begin{enumerate}
\item There is a bijection between the vertex set $V$ of $\Gamma$ and the fixed point set $M^{S^1}$ of $M$.
\item The sign of any vertex is equal to the sign of its corresponding fixed point.
\item For any vertex $v$, the multiset of the labels of the edges of $v$ is equal to the multiset of the weights of its corresponding fixed point.
\item For every edge $e \in E$, the vertices of $e$ lie in the same connected component of $M^{\mathbb{Z}_{w(e)}}$.
\end{enumerate}
\end{Definition}

In particular, if a multigraph $\Gamma$ describes $M$, then $\Gamma$ is $n$-regular, where $\dim M=2n$.

\begin{rem}
We will refer to the weights of a vertex as the weights of the corresponding element of $M^{S^1}$.
\end{rem}

\begin{rem}
Let $M$ be a compact oriented manifold endowed with an $S^1$-action with a discrete fixed point set. Suppose that each isotropy submanifold of $M$ containing an $S^1$-fixed point is orientable. In this case, the proof of \cite[Proposition 2.6]{J3} shows that we can construct a multigraph describing $M$. On the other hand, a multigraph describing $M$ need not be unique.
\end{rem}

We discuss how \cite[Proposition 2.6]{J3} associated a signed labeled multigraph to a compact oriented $S^1$-manifold with a discrete fixed point, whose isotropy submanifolds are orientable.

Let the circle act on a compact oriented manifold $M$ with a discrete fixed point set. Suppose that each isotropy submanifold is orientable. First, to each fixed point we assign a vertex. Let $w$ be a weight at a fixed point $p$ and let $F$ be a component of $M^{\mathbb{Z}_w}$ that contains $p$. By assumption, $F$ is orientable; choose an orientation of $F$. 
The circle acting on $M$ also acts on $F$ as a restriction, and the smallest weight of this action is $w$. Applying Lemma \ref{l25} to this action on $F$, the multiplicity of weight $w$ at fixed points with sign $+1$ (for this action on $F$) is equal to the multiplicity of $w$ at fixed points with sign $-1$ (for this action on $F$). Thus, we can draw edges of label $w$ between fixed points $q \in F^{S^1}=F \cap M^{S^1}$ with different signs for this $S^1$-action on $F$. This is the idea of the proof of \cite[Proposition 2.6]{J3}.

\begin{pro} \label{p210} \cite[Proposition 2.6]{J3}
Let the circle act on a $2n$-dimensional compact oriented manifold $M$ with a discrete fixed point set. Suppose that each isotropy submanifold is orientable. Then there exists a (signed labeled) multigraph describing $M$ that has no self-loops. \end{pro}

\begin{rem}
For torus actions on manifolds with discrete fixed point sets, certain multigraphs associated to manifolds have proved useful in the study of fixed point problems. Such a multigraph encodes the data on the fixed point set of a manifold: a vertex of a multigraph corresponds to a fixed point, and the label of an edge corresponds to a weight at a fixed point. If a given manifold $M$ admits an almost complex structure preserved by the action (for example, $M$ is a symplectic manifold), we associate a directed labeled multigraph \cite{GS}, \cite{JT}, \cite{T}; also see \cite{J6} for an extension to torus actions and \cite{J5} for an extension to unitary manifolds. 
\end{rem}

\section{Pontryagin classes}\label{s3}

In this section, for a compact oriented $S^1$-manifold with a discrete fixed point set, we prove a few properties in terms of Pontryagin classes. We will use these to prove Theorem \ref{t11}.

For an $S^1$-equivariant oriented (complex) vector bundle $E \to M$, the equivariant Euler class and the equivariant Pontryagin class (the equivariant Chern class) of $E$ are the Euler class and the Pontryagin class (the Chern class) of the bundle $E \times_{S^1} S^\infty \to M \times_{S^1} S^\infty$. If the circle acts on a manifold $M$, the action acts on the tangent bundle $TM$ of $M$, and we consider the $S^1$-equivariant vector bundle $TM \to M$.

Suppose the circle group acts on a $2n$-dimensional compact oriented manifold $M$ and the fixed point set is non-empty and finite. Let $p$ be an isolated fixed point. The tangent space $T_pM$ to $M$ at $p$ has a decomposition
\begin{center}
$T_pM=L_{p,1} \oplus \cdots \oplus L_{p,n}$
\end{center}
where the circle acts on each $L_{p,i}$ with weight $w_{p,i}>0$. Each $L_{p,i}$ can be given a complex structure. The total equivariant Chern class of $T_pM$ is $\prod_{i=1}^n (1 \pm w_{p,i} x)$, where each sign depends on the choice of a complex structure on $L_{p,i}$. Here, $x$ is a degree 2 generator of $H^*(\mathbb{CP}^{\infty})$. Moreover, the normal bundle to $p$ is the tangent space $T_pM$ to $p$, and thus the equivariant Euler class of the normal bundle to $p$ is
\begin{center}
$e_{S^1}(N_p)=e_{S^1}(T_pM)=c_n^{S^1}(T_pM)=\epsilon(p) \prod_{j=1}^n (w_{p,j} x)$.
\end{center}
Here, $c_i^{S^1}$ denotes the $i$-th equivariant Chern class of $M$.
The Chern classes $c_i$ determine the Pontryagin classes $P_i$ by
\begin{center}
$1-P_1+P_2-+\cdots\pm P_n=(1-c_1+c_2-+\cdots \pm c_n)(1+c_1+c_2+\cdots+c_n)$,
\end{center}
see \cite{MS}. This shows the following.

\begin{pro} \label{p31}
Let the circle act on a $2n$-dimensional compact oriented manifold $M$ with a discrete fixed point set. Then the restriction of the total equivariant Pontryagin class $P$ of $M$ to a fixed point $p$ is given by
\begin{center}
$\displaystyle P(p)=\prod_{i=1}^n (1+w_{p,i}^2 x^2)$,
\end{center}
where $x$ is a generator of $H^*(\mathbb{CP}^{\infty})$.
\end{pro}

Let the circle on a $2n$-dimensional compact oriented manifold $M$ with $k$ fixed points, $0 < k < \infty$. Fix a positive integer $m$. Denote by $B=\{B_1, \cdots, B_s\}$ the set $\{P_m(p) \, | \, p \in M^{S^1}\}$, where $P_m(p)$ is the $m$-th equivariant Pontryagin class $P_m$ of $M$ at $p$. Let 
$$\displaystyle A_i = \sum_{P_m(p)=B_i} \epsilon(p) \frac{1}{\prod_{j=1}^n w_{p,j}}$$
for $1 \leq i \leq s$. In the proof of Theorem \ref{t11} we will use the fact that if $\dim M \geq 12$ and there are exactly 3 fixed points, the restrictions of the first equivariant Pontryagin class at the fixed points are the same. This will follow from the lemmas below.

\begin{lem} \label{l32}
Let the circle on a $2n$-dimensional compact oriented manifold $M$ with $k$ fixed points, $0 < k < \infty$. Fix a positive integer $m$. 
Suppose that there exists some $i$ such that $A_i \neq 0$. Then $\dim M < 4km$.
\end{lem}

\begin{proof} Assume on the contrary that $\dim M \geq 4km$. Note that $s \leq k$. By Theorem \ref{t21}, for $0 \leq l < s$,
\begin{center}
$\displaystyle 0=\int_M (P_m)^l = \sum_{p \in M^{S^1}} \epsilon(p) \frac{(P_m (p))^l}{\prod_{j=1}^n w_{p,j}} = \sum_{i=1}^s \sum_{p \in M^{S^1}, P_m(p)=B_i} \epsilon(p) \frac{(P_m(p))^l}{\prod_{j=1}^n w_{p,j}}$ 

$\displaystyle = \sum_{i=1}^s \sum_{p \in M^{S^1}, P_m(p)=B_i} \epsilon(p) \frac{B_i^l}{\prod_{j=1}^n w_{p,j}}$ 

$\displaystyle = \sum_{i=1}^s B_i^l \sum_{p \in M^{S^1}, P_m(p)=B_i} \epsilon(p) \frac{1}{\prod_{j=1}^n w_{p,j}}=\sum_{i=1}^s B_i^l A_i$.
\end{center}

Let $C$ be a matrix $(C)_{ij} = B_i^{j-1}$, $1 \leq i , j \leq s$. Then we have a homogeneous system of linear equations
\begin{center}
$C \cdot A =0$,
\end{center}
where $A=(A_1,\cdots,A_s)^t$. Because $B_1$, $\cdots$, $B_s$ are all distinct, $C$ is a Vandermonde matrix. Thus $\det(C) \neq 0$. Therefore, $A_i=0$ for all $i$, but this contradicts the assumption. \end{proof}

As an application of Lemma \ref{l32}, we obtain the following lemma.

\begin{lem} \label{l33}
Let the circle on a compact oriented manifold $M$ with $k$ fixed points, $0 < k < \infty$. Suppose that $\textrm{sign}(M)=\pm(k-2)$. If $\dim M \geq 4km$, then the restrictions of the equivariant $m$-th Pontryagin class $P_m$ of $M$ at all fixed points are the same. That is, $P_m(p_1)=P_m(p_2)$ for any $p_1, p_2 \in M^{S^1}$.
\end{lem}

\begin{proof}
First, suppose that $\textrm{sign}(M)=k-2$. Since $\mathrm{sign}(M)=\sum_{p \in M^{S^1}} \epsilon(p)$, $k-1$ fixed points have sign $+1$ and one fixed point has sign $-1$. Let $B$ and $A_i$ be as above. By Lemma \ref{l32}, $A_i=0$ for all $i$. Because $A_i = \sum_{P_m(p)=B_i} \epsilon(p) \frac{1}{\prod_{j=1}^n w_{p,j}}$, for $A_i$ to vanish, there must be at least two fixed points $p_1$ and $p_2$ with different signs and $P_m(p_1)=B_i=P_m(p_2)$. Since there is only one fixed point with sign $-1$, this implies that $s=1$ and hence $B=\{B_1\}$. The case where $\textrm{sign}(M)=-k+2$ is similar.
\end{proof}

If $\dim M \geq 12m$ and there are exactly 3 fixed points, the above lemma implies the following.

\begin{lem} \label{l34}
Let the circle on a compact oriented manifold $M$ with exactly three fixed points. If $\dim M \geq 12m$, the restrictions of the equivariant $m$-th Pontryagin class $P_m$ of $M$ at all fixed points are the same. That is, $P_m(p_1)=P_m(p_2)$ for any $p_1, p_2 \in M^{S^1}$.
\end{lem}

\begin{proof}
Since the action has 3 fixed points, Lemma \ref{l23} and the fact that $\mathrm{sign}(M)=\sum_{p \in M^{S^1}} \epsilon(p)$ imply that $\textrm{sign}(M)=\pm 1$. Therefore, this lemma follows from Lemma \ref{l33}. \end{proof}

\section{Preliminaries: 3 fixed points} \label{s4}

From now on, we focus on the case of three fixed points. In this section, we analyze a multigraph describing a compact oriented $S^1$-manifold with 3 fixed points, and use this to derive relationships between the weights at different fixed points.

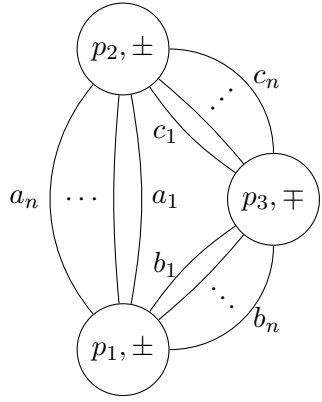
\begin{figure}
\centering
\begin{tikzpicture}[state/.style ={circle, draw}]
\node[vertex] (b) at (3, 0) {$p_{3},\mp$};
\node[vertex] (c) at (0, 3) {$p_{2},\pm$};
\node[vertex] (a) at (0, -3) {$p_{1},\pm$};
\path (a) [bend left =10] edge node[left] {$b_1$} (b);
\path (a) [bend right =5] edge node[right, sloped, rotate=270] {$\cdots$} (b);
\path (a) [bend right =35] edge node[right] {$b_n$} (b);
\path (b) [bend left =10] edge node[left] {$c_1$} (c);
\path (b) [bend right =5] edge node[right, sloped, rotate=90] {$\cdots$} (c);
\path (b) [bend right =35] edge node[right] {$c_n$} (c);
\path (a) [bend right =10] edge node[right] {$a_1$} (c);
\path (a) [bend left =5] edge node[left] {$\cdots$} (c);
\path (a) [bend left =30] edge node[left] {$a_n$} (c);
\end{tikzpicture}
\caption{Multigraph describing $M$}\label{fig}
\end{figure}

\begin{pro} \label{p41}
Let the circle act on a compact oriented manifold $M$ with three fixed points. Suppose that each isotropy submanifold is orientable. Then Figure \ref{fig} describes $M$.
In particular, the fixed point data of $M$ is
\begin{center}
$\Sigma_{p_1}=\{\pm, a_1,\cdots,a_n,b_1,\cdots,b_n\}$

$\Sigma_{p_2}=\{\pm, a_1,\cdots,a_n,c_1,\cdots,c_n\}$

$\Sigma_{p_3}=\{\mp, b_1,\cdots,b_n,c_1,\cdots,c_n\}$
\end{center}
for some positive integers $a_i$, $b_i$, and $c_i$, $1 \leq i \leq n$, such that $a_1 \leq \cdots \leq a_n$, $b_1 \leq \cdots \leq b_n$, and $c_1 \leq \cdots \leq c_n$, where $p_1$, $p_2$, and $p_3$ are the three fixed points.
\end{pro}

\begin{proof}
By Proposition \ref{p210}, there exists a signed labeled multigraph $\Gamma$ describing $M$ that has no self-loops. By Lemma \ref{c27}, $\dim M=4n$ for some positive integer $n$. Let $p_1$, $p_2$, and $p_3$ denote the fixed points. Each fixed point has $2n$ weights. Suppose that there are $m$-edges between $p_1$ and $p_2$. Since $\Gamma$ has no self-loops, there are $2n-m$ edges between $p_1$ and $p_3$ and hence there are $2n-m=m$ edges between $p_2$ and $p_3$. Therefore, $m=n$.

Since $\textrm{sign}(M)=\epsilon(p_1)+\epsilon(p_2)+\epsilon(p_3)$, with Lemma \ref{l23} it follows that $\textrm{sign}(M)=\pm 1$ and two fixed points have the same sign, say $p_1$ and $p_2$. This completes the proof. \end{proof}

\begin{rem}
In Figure \ref{fig}, there is an asymmetry between $a_i$'s and $b_i$'s, $c_i$'s. An edge with label $a_i$ has vertices with same signs, while an edge with label $b_i$ or $c_i$ has vertices with different signs.
\end{rem}

From Proposition \ref{p41} we derive relations on the weights at the fixed points that many weights are equal to sums of other weights.

\begin{pro} \label{p42}
In Proposition \ref{p41}, the following two multisets are equal.
\begin{center}
$\displaystyle A=\{\sum_{i=1}^n (d_i a_i+e_i b_i+f_i c_i) \, | \, 0 \leq d_i,e_i,f_i \leq 1,\sum_{i=1}^n d_i \equiv 1 \mod 2, \sum_{i=1}^n e_i \equiv 0 \mod 2, \sum_{i=1}^n f_i \equiv 0 \mod 2 \}$

$\displaystyle B=\{\sum_{i=1}^n (d_i a_i+e_i b_i+f_i c_i) \, | \, 0 \leq d_i,e_i,f_i \leq 1, \sum_{i=1}^n d_i \equiv 0 \mod 2, \sum_{i=1}^n e_i \equiv 1 \mod 2, \sum_{i=1}^n f_i \equiv 1 \mod 2 \}$
\end{center}
Here, $d_i$, $e_i$, and $f_i$ are integers. Moreover, 
\begin{itemize}
\item $a_1=b_1+c_1$; and 
\item $a_2=b_1+c_2$ or $a_2=b_2+c_1$. 
\end{itemize}
\end{pro}

\begin{proof}
By Theorem \ref{t22}, the signature of $M$ satisfies
\begin{center}
$\displaystyle \textrm{sign}(M)=\pm 1=\pm \prod_{i=1}^n \frac{ (1+t^{a_i})(1+t^{b_i})}{(1-t^{a_i})(1-t^{b_i})} \pm \prod_{i=1}^n \frac{ (1+t^{a_i})(1+t^{c_i})}{(1-t^{a_i})(1-t^{c_i})} \mp \prod_{i=1}^n \frac{(1+t^{b_i})(1+t^{c_i})}{(1-t^{b_i})(1-t^{c_i})}$
\end{center}
and hence
\begin{center}
$\displaystyle 1=\prod_{i=1}^n \frac{ (1+t^{a_i})(1+t^{b_i})}{(1-t^{a_i})(1-t^{b_i})}+\prod_{i=1}^n \frac{ (1+t^{a_i})(1+t^{c_i})}{(1-t^{a_i})(1-t^{c_i})}-\prod_{i=1}^n \frac{(1+t^{b_i})(1+t^{c_i})}{(1-t^{b_i})(1-t^{c_i})}$
\end{center}
for all indeterminates $t$. We multiply the above equation by the least common multiple of the denominators $\prod_{i=1}^n [(1-t^{a_i})(1-t^{b_i})(1-t^{c_i})]$ to have
\begin{center}
$\displaystyle \prod_{i=1}^n [(1-t^{a_i})(1-t^{b_i})(1-t^{c_i})]=\prod_{i=1}^n[ (1+t^{a_i})(1+t^{b_i})(1-t^{c_i})]+\prod_{i=1}^n [(1+t^{a_i})(1-t^{b_i})(1+t^{c_i})]-\prod_{i=1}^n [(1-t^{a_i})(1+t^{b_i})(1+t^{c_i})]$,
\end{center}
and hence
\begin{equation} \label{e1}
\begin{split}
\displaystyle 0=-\prod_{i=1}^n [(1-t^{a_i})(1-t^{b_i})(1-t^{c_i})]+\prod_{i=1}^n [(1+t^{a_i})(1+t^{b_i})(1-t^{c_i})] \\
\displaystyle +\prod_{i=1}^n [(1+t^{a_i})(1-t^{b_i})(1+t^{c_i})]-\prod_{i=1}^n (1-t^{a_i})(1+t^{b_i})(1+t^{c_i})].
\end{split}
\end{equation}

Fix integers $0 \leq d_i,e_i,f_i \leq 1$ for $1 \leq i \leq n$. The coefficient of the term $\displaystyle t^{\sum_{i=1}^n (d_i a_i+e_i b_i+f_i c_i)}$ in the right hand side of Equation \eqref{e1} is
\begin{enumerate}[(i)]
\item 4 if $\displaystyle \sum_{i=1}^n d_i \equiv 1 \mod 2, \sum_{i=1}^n e_i \equiv 0 \mod 2, \sum_{i=1}^n f_i \equiv 0 \mod 2$
\item -4 if $\displaystyle \sum_{i=1}^n d_i \equiv 0 \mod 2, \sum_{i=1}^n e_i \equiv 1 \mod 2, \sum_{i=1}^n f_i \equiv 1 \mod 2$.
\item 0 otherwise.
\end{enumerate}
For instance, if $\sum_{i=1}^n d_i \equiv 1 \mod 2$, $\sum_{i=1}^n e_i \equiv 0 \mod 2$, and  $\sum_{i=1}^n f_i \equiv 1 \mod 2$, then the coefficient of $\displaystyle t^{\sum_{i=1}^n (d_i a_i+e_i b_i+f_i c_i)}$ in the first, second, third, fourth summand of the right hand side of Equation \eqref{e1} is $-1$, $-1$, $+1$, $+1$, respectively, and thus the coefficient of the term $\displaystyle t^{\sum_{i=1}^n (d_i a_i+e_i b_i+f_i c_i)}$ in the right hand side of Equation \eqref{e1} is 0.

Since the right hand side of Equation \eqref{e1} must be zero, it follows that $A=B$ as multisets.

The smallest element in $A$ is $a_1$ and the smallest element in $B$ is $b_1+c_1$. Hence $a_1=b_1+c_1$. 

Next, we consider the next smallest element in $A$ and that in $B$, i.e., the smallest element in $A \setminus \{a_1\}$ and that in $B \setminus \{b_1+c_1\}$. Possible candidates in $A \setminus \{a_1\}$ are $a_2$, $a_1+b_1+b_2$, and $a_1+c_1+c_2$, and those in $B \setminus \{b_1+c_1\}$ are $b_1+c_2$, $b_2+c_1$, and $a_1+a_2+b_1+c_1$. Because $a_1=b_1+c_1$, it follows that $a_1+b_1+b_2=2b_1+b_2+c_1>b_2+c_1$ and $a_1+c_1+c_2=b_1+2c_1+c_2>b_1+c_2$. Hence both $a_1+b_1+b_2$ and $a_1+c_1+c_2$ cannot be the smallest element in $A \setminus \{a_1\}$ and thus the smallest element in $A \setminus \{a_1\}$ is $a_2$. Next, $a_2$ is smaller than all elements in $B \setminus \{b_1+c_1\}$ of the form $\sum_{i=1}^n (d_i a_i+e_i b_i+f_i c_i)$ with $\sum_{i=1}^n d_i \geq 2$. Therefore, $a_2$ has to equal an element $\sum_{i=1}^n (d_i a_i+e_i b_i+f_i c_i)$ in $B \setminus \{b_1+c_1\}$ such that $\sum_{i=1}^n d_i=0$. Among elements $\sum_{i=1}^n (d_i a_i+e_i b_i+f_i c_i)$ in $B \setminus \{b_1+c_1\}$ with $\sum_{i=1}^n d_i=0$, the smallest element is $b_1+c_2$ or $b_2+c_1$. Therefore, $a_2=b_1+c_2$ or $a_2=b_2+c_1$. \end{proof}

\section{Classification: 8 dimensional manifolds with 3 fixed points} \label{s5}

In this section, we classify the fixed point data of an 8-dimensional compact oriented $S^1$-manifold with 3 fixed points, whose isotropy submanifolds are orientable.

\begin{proof}[\textbf{Proof of Theorem \ref{t12}}]
By Proposition \ref{p41}, the fixed point data of $M$ is
\begin{center}
$\Sigma_{p_1}=\{\pm, a_1, a_2, b_1, b_2\}$

$\Sigma_{p_2}=\{\pm, a_1, a_2, c_1, c_2\}$

$\Sigma_{p_3}=\{\mp, b_1, b_2, c_1, c_2\}$
\end{center}
for some positive integers $a_i$, $b_i$, and $c_i$, $1 \leq i \leq 2$, such that $a_1 \leq a_2$, $b_1 \leq b_2$, and $c_1 \leq c_2$. By Proposition \ref{p42}, $a_1=b_1+c_1$, and  $a_2=b_1+c_2$ or $a_2=b_2+c_1$. Without loss of generality, by permuting $p_1$ and $p_2$ (by permuting the role of $b_i$ and $c_i$) if necessary, we may assume that $a_2=b_1+c_2$.

By Theorem \ref{t21}, we push-forward the equivariant cohomology class $1$ to obtain
\begin{center}
$\displaystyle 0=\int_M 1=\pm \frac{1}{a_1 a_2 b_1 b_2} \pm \frac{1}{a_1 a_2 c_1 c_2} \mp \frac{1}{b_1 b_2 c_1 c_2}$
\end{center}
and hence
\begin{center}
$0=c_1 c_2 + b_1 b_2 - a_1 a_2=c_1 c_2 + b_1 b_2 - (b_1+c_1)(b_1+c_2)$.
\end{center}
It follows that $b_2=b_1+c_1+c_2$. Theorem \ref{t12} follows if we set $b_1=a$, $c_1=b$, and $c_2=c$.
\end{proof}

\begin{rem}
In the proof of Theorem \ref{t12}, the conclusion that $b_2=b_1+c_1+c_2$ also follows from Proposition \ref{p42}.
\end{rem}

We give an example of an $S^1$-action on $\mathbb{HP}^2$ that has 3 fixed points. By Theorem \ref{t12}, for any circle action on an 8-dimensional compact oriented manifold with 3 fixed points, if its isotropy submanifolds are orientable, the manifold has the same fixed point data as the fixed point data of this example, up to reversing the orientation of $\mathbb{HP}^2$.

\begin{exa} \label{e214}
Let $s=s_1+s_2 i + s_3 j + s_4 k$ denote an element of the quaternions $\mathbb{H}$. Under an embedding of $S^1 \subset \mathbb{C}$ into $\mathbb{H}$ by $g=s_1 + s_2 i \hookrightarrow s_1 + s_2 i + 0 j + 0 k$, let the circle act on the quaternionic projective space $\mathbb{HP}^2=(\mathbb{H}^3 \setminus \{(0,0,0)\})/((x_0,x_1,x_2) \sim (s x_0, s x_1, s x_2))$ by
\begin{center}
$g \cdot [x_0:x_1:x_2]=[g^d x_0:g^{e} x_1:g^{f} x_2]$
\end{center}
for all $g \in S^1$, where $d<e<f$ are either all positive integers, or all positive half integers. That is, $(d,e,f) \in \mathbb{N}^3$ or $(d,e,f) \in (\mathbb{N}+1/2)^3$. There are 3 fixed points $[1:0:0]$, $[0:1:0]$, and $[0:0:1]$. 

At $[1:0:0]$, using local coordinates $(\frac{x_1}{x_0},\frac{x_2}{x_0})$, the circle acts as
\begin{center}
$g \cdot (y_1,z_1,y_2,z_2)=(g^{e-d} y_1, g^{e+d} z_1, g^{f-d} y_2, g^{f+d} z_2)$,
\end{center}
for all $g \in S^1$, where $\frac{x_i}{x_0}=(y_i,z_i)$, $y_i,z_i \in \mathbb{C}$, $1 \leq i \leq 2$. Therefore, the fixed point data of $[1:0:0]$ is $\{+,e-d,e+d,f-d,f+d\}$. Similar computations show that the fixed point data of $\mathbb{HP}^2$ is
\begin{center}
$\{+, |e+d|, |e-d|, |f+d|, |f-d|\}$, 

$\{-, |d+e|, |d-e|, |f+e|, |f-e|\}$, 

$\{+, |d+f|, |d-f|, |e+f|, |e-f|\}$.
\end{center}
If we set $d=\frac{c-b}{2}$, $e=\frac{b+c}{2}$, and $f=a+\frac{b+c}{2}$, then this is equivalent to the fixed point data of Theorem \ref{t12} up to reversing orientation. \end{exa}

\section{Further preliminaries: 12 dimension and 3 fixed points} \label{s6}

In this section, to prove Theorem \ref{t11}, we investigate further properties that a 12-dimensional compact oriented $S^1$-manifold with 3 fixed points satisfies.

\begin{lem}\label{l61}
In Proposition \ref{p41}, if $n \geq 3$, then $\sum_{i=1}^n a_i^2=\sum_{i=1}^n b_i^2=\sum_{i=1}^n c_i^2$.
\end{lem}

\begin{proof}
This follows from Proposition \ref{p31} and Lemma \ref{l34}.
\end{proof}

In the setting of Proposition \ref{p41}, Lemma \ref{l61} implies that $a_3$ cannot be the biggest weight.

\begin{lem} \label{l62}
In Proposition \ref{p41}, suppose that $n=3$. Then $a_3$ cannot be the biggest weight.
\end{lem}

\begin{proof}
By Proposition \ref{p42}, $a_1=b_1+c_1$, and $a_2=b_1+c_2$ or $a_2=b_2+c_1$. Hence $a_1$ is bigger than both $b_1$ and $c_1$, and $a_2$ bigger than $b_2$ or $c_2$. By Lemma \ref{l61}, $a_1^2+a_2^2+a_3^2=b_1^2+b_2^2+b_3^2=c_1^2+c_2^2+c_3^2$. Hence this lemma holds. \end{proof}

In \cite{J3}, the author classified an $S^1$-manifold whose weights at the fixed points are the same. We shall use this to obtain further results on the possibilities for $a_i$, $b_i$, $c_i$ in Proposition \ref{p41}, as in Lemma \ref{l64} below.

\begin{theorem} \label{t63} \cite{J3}
Let the circle act on a $2n$-dimensional compact oriented manifold $M$ with a discrete fixed point set. Suppose that the weights at each fixed point are $\{a_1,\cdots,a_n\}$ for each $a_i$ a positive integer. Then the number of fixed points $p$ with $\epsilon(p)=+1$ and that with $\epsilon(p)=-1$ are equal. In particular, the total number of fixed points is even and the signature of $M$ vanishes.
\end{theorem}

Theorem \ref{t63} applies to show a property on a component of $M^{\mathbb{Z}_{b_3}}$.

\begin{lem} \label{l64}
In Proposition \ref{p41} suppose that $n=3$ and $b_3$ is the biggest weight. Then a component of $M^{\mathbb{Z}_{b_3}}$ containing $p_1$ and $p_3$ does not contain $p_2$. In particular, none of $a_i$ and $c_i$ are equal to $b_3$.
\end{lem}

\begin{proof}
Since Figure \ref{fig} describes $M$ and there is an edge with label $b_3$ between $p_1$ and $p_3$, these two fixed points $p_1$ and $p_3$ are in the same component $F$ of $M^{\mathbb{Z}_{b_3}}$, which is compact. To prove this lemma, suppose on the contrary that $F$ also contains $p_2$. The circle acting on $M$ acts on $F$ as a restriction and this $S^1$-action on $F$ has fixed points $p_1$, $p_2$, and $p_3$. At a fixed point $p \in F^{S^1}$ of this $S^1$-action on $F$, every weight in the tangent space $T_pF$ to $F$ at $p$ is a multiple of $b_3$; since $b_3$ is the biggest weight every weight in $T_pF$ is $b_3$. By assumption, $F$ is orientable; choose an orientation of $F$. Applying Theorem \ref{t63} to this $S^1$-action on $F$, the $S^1$-action on $F$ must have an even number of fixed points, which is a contradiction. Therefore, $F$ does not contain $p_2$. Consequently, none of $a_i$ and $c_i$ are equal to $b_3$. \end{proof}

To prove Theorem \ref{t11}, we need a technical lemma on a relationship between the weights at fixed points lying in the same component of $M^{\mathbb{Z}_w}$, where $w$ is the biggest weight. For doing so, we need more definitions.

Let the circle act on a $2n$-dimensional compact oriented manifold $M$ with a discrete fixed point set. Let $w$ be a positive integer. Let $F$ be a connected component of $M^{\mathbb{Z}_w}$ such that $F \cap M^{S^1} \neq \emptyset$. Suppose that $F$ is orientable. Choose an orientation of $F$. Take an orientation of the normal bundle $NF$ of $F$ so that the induced orientation on $NF \oplus TF$ agrees with the orientation of $M$. Let $p\in F \cap M^{S^1}$ be an $S^1$-fixed point. By permuting $L_{p,i}$'s, we may let
\begin{center}
$T_pM=L_{p,1} \oplus \cdots \oplus L_{p,m} \oplus L_{p,m+1} \oplus \cdots \oplus L_{p,n}$,
\end{center}
where $N_pF=L_{p,1} \oplus \cdots \oplus L_{p,m}$ and $T_pF=L_{p,m+1} \oplus \cdots \oplus L_{p,n}$, and the circle acts on each $L_{p,i}$ with weight $w_{p,i}$. As in the introduction, we orient each $L_{p,i}$ so that $w_{p,i}$ is positive.

\begin{Definition} \label{d65} 
\begin{enumerate}[(1)]
\item $\epsilon_F(p)=+1$ if the orientation on $F$ agrees with the orientation on $L_{p,m+1}\oplus \cdots \oplus L_{p,n}$, and $\epsilon_F(p)=-1$ otherwise.
\item $\epsilon_N(p)=+1$ if the orientation on $NF$ agrees with the orientation on $L_{p,1}\oplus \cdots \oplus L_{p,m}$, and $\epsilon_N(p)=-1$ otherwise.
\end{enumerate} 
\end{Definition}

By definition, $\epsilon(p)=\epsilon_F(p) \cdot \epsilon_N(p)$.

\begin{lem} \label{l66}
In Proposition \ref{p41}, suppose that $n=3$ and $b_3$ is the biggest weight. Then there exist $i_1,i_2,i_3$ and $j_1,j_2,j_3$ such that $a_{i_1}+c_{j_1}=b_3$, $a_{i_2}+c_{j_2}=b_3$, $a_{i_3}=c_{j_3}$, and $c_1 \neq c_{j_3}$, where $\{i_1,i_2,i_3\}=\{j_1,j_2,j_3\}=\{1,2,3\}$.
\end{lem}

\begin{proof}
By Lemma \ref{l64}, a component $F$ of $M^{\mathbb{Z}_{b_3}}$ containing $p_1$ and $p_3$ does not contain $p_2$. Thus, if we restrict the circle action on $M$ to a circle action on $F$, this $S^1$-action on $F$ has exactly two fixed points $p_1$ and $p_3$. By assumption, $F$ is orientable; choose an orientation of $F$ and also take an orientation of the normal bundle $NF$ of $F$ so that the induced orientation on $NF \oplus TF$ is the orientation of $M$. If we apply Theorem \ref{t213} below to the $S^1$-action on $F$, it follows that $\epsilon_F(p_1)=-\epsilon_F(p_3)$. Since $\epsilon(p_1)=-\epsilon(p_3)$, with $\epsilon(p_1)=\epsilon_F(p_1) \cdot \epsilon_N(p_1)$ and $\epsilon(p_3)=\epsilon_F(p_3) \cdot \epsilon_N(p_3)$, this implies that $\epsilon_N(p_1)=\epsilon_N(p_3)$.

Let $k$ be the biggest $k$ such that $b_k < b_3$.
Let $a_1$, $a_2$, $a_3$, $b_1$, $\cdots$, $b_k$ ($c_1$, $c_2$, $c_3$, $b_1$, $\cdots$, $b_k$) be the weights of $L_{p_1,1}\oplus \cdots \oplus L_{p,1,m}$ ($L_{p_3,1}\oplus \cdots \oplus L_{p,3,m}$), respectively.

Let $-L_{p_3,i}$ denote $L_{p_3,i}$ with opposite orientation. Since $NF$ is an oriented $\mathbb{Z}_{b_3}$-bundle over $F$ and $F$ is connected, the $\mathbb{Z}_{b_3}$-representations of $N_{p_1}F$ and $N_{p_3}F$ are isomorphic, which are given by $C:=L_{p_1,1}\oplus \cdots \oplus L_{p,1,m}$ and $D:=L_{p_3,1}\oplus \cdots \oplus L_{p,3,m}$, respectively.

Suppose that $\phi_1$ is an orientation preserving $\mathbb{Z}_{b_3}$-isomorphism from $C$ to $D$. Suppose that $\phi_1$ takes $L_{p_1,4}$ to some $\pm L_{p_3,i}$, where $i \leq 3$. Then $\phi_1$ takes some $L_{p_1,j}$ to $\pm L_{p_3,4}$. Let $\phi_2$ be an orientation preserving isomorphism from $D$ to itself that intertwines $L_{p_3,i}$ and $\pm L_{p_3,4}$. Then $\phi_2 \circ \phi_1$ is an orientation preserving isomorphism from $C$ to $D$ that takes $L_{p_1,4}$ to $L_{p_3,4}$. Therefore, we may assume that an isomorphism from $C$ to $D$ takes $L_{p_1,i+3}$ to $L_{p_3,i+3}$, where $0 \leq i \leq k$. This implies that $L_{p_1,1}\oplus L_{p_1,2} \oplus L_{p_1,3}$ and $L_{p_3,1}\oplus L_{p_3,2} \oplus L_{p_3,3}$ are isomorphic as $\mathbb{Z}_{b_3}$-representations.

This means that one of the following holds:
\begin{enumerate}[(1)]
\item There is a bijection $\sigma:\{1,2,3\} \to \{1,2,3\}$ such that $L_{p_1,i}$ is isomorphic to $L_{p_3,\sigma(i)}$ for all $1 \leq i \leq 3$.
\item There is a bijection $\sigma:\{1,2,3\} \to \{1,2,3\}$ such that for exactly one $i$, say $i_3$, $L_{p_1,i_3}$ is isomorphic to $L_{p_3,\sigma(i_3)}$, and for $j$ such that $j \neq i_3$, $L_{p_1,j}$ is isomorphic to $-L_{p_3,\sigma(j)}$.
\end{enumerate}

Case (1) means that $a_i \equiv c_{\sigma(i)} \mod b_3$ for all $i$. Because $b_3$ is the biggest weight, this means that $a_i= c_{\sigma(i)}$ for all $i$. But this cannot hold since $c_1<a_1 \leq a_2 \leq a_3$ by Proposition \ref{p42}. Therefore, Case (2) holds. 

Now, $L_{p_1,i_3}$ being isomorphic to $L_{p_3,\sigma(i_3)}$ with given orientation means that $a_{i_3} \equiv c_{\sigma(i_3)} \mod b_3$. Therefore, $a_{i_3}=c_{\sigma(i_3)}$. Since $c_1<a_1\leq a_2 \leq a_3$ by Proposition \ref{p42}, $c_1$ cannot be $c_{\sigma(i_3)}$.

Let $j \neq i_3$. For this $j$, $L_{p_1,j}$ being isomorphic to $L_{p_3,\sigma(j)}$ with opposite orientation means that $a_j \equiv -c_{\sigma(j)} \mod b_3$. Since $a_i,c_i<b_3$ for all $i$, this implies that $a_{j}+c_{\sigma(j)}=b_3$. Setting $j_k=\sigma(i_k)$, this completes the proof. \end{proof}

If a circle action on a compact oriented manifold has exactly 2 fixed points, Kosniowski proved that the fixed point data is the same as that of a rotation of $S^{2n}$ \cite{K3}. This also easily follows from the index formula of Theorem \ref{t22}.

\begin{theorem} \label{t213} \cite{K3}
Let the circle act on a compact oriented manifold with two fixed points $p$ and $q$. Then the weights at $p$ and $q$ agree up to order and $\epsilon(p)=-\epsilon(q)$. \end{theorem}

\begin{proof}
By Lemma \ref{l23}, $\mathrm{sign}(M)=0$ and $\epsilon(p)=-\epsilon(q)$. By Theorem \ref{t22},
\begin{center}
$\displaystyle 0=\textrm{sign}(M) = \pm \prod_{i=1}^{n} \frac{1+t^{w_{p,i}}}{1-t^{w_{p,i}}} \mp \prod_{i=1}^{n} \frac{1+t^{w_{q,i}}}{1-t^{w_{q,i}}}$
\end{center}
for all indeterminate $t$. Therefore, $w_{p,i}$ and $w_{q,i}$ agree up to order. \end{proof}

\section{Non-existence: 12 dimension and 3 fixed points} \label{s7}

With all of the above, we are ready to prove our main result, Theorem \ref{t11}.

\begin{proof}[\textbf{Proof of Theorem \ref{t11}}]

Assume on the contrary that there exists a circle action on a 12-dimensional compact oriented manifold $M$ with exactly 3 fixed points, whose isotropy submanifolds are orientable. By Proposition \ref{p41}, Figure \ref{fig} describes $M$. In particular, the fixed point data of $M$ is
\begin{center}
$\Sigma_{p_1}=\{\pm, a_1,a_2,a_3,b_1,b_2,b_3\}$

$\Sigma_{p_2}=\{\pm, a_1,a_2,a_3,c_1,c_2,c_3\}$

$\Sigma_{p_3}=\{\mp, b_1,b_2,b_3,c_1,c_2,c_3\}$
\end{center}
for some positive integers $a_i$, $b_i$, and $c_i$, $1 \leq i \leq 3$, such that $a_1 \leq a_2 \leq a_3$, $b_1 \leq b_2 \leq b_3$, and $c_1 \leq c_2 \leq c_3$, where $p_1$, $p_2$, and $p_3$ are the three fixed points.

By Lemma \ref{l62}, $a_3$ cannot be the biggest weight. Without loss of generality, by permuting $p_1$ and $p_2$ (by permuting the role of $b_i$ and $c_i$) if necessary, we may assume that $b_3$ is the biggest weight.

By Lemma \ref{l66}, $a_{i_1}+c_{j_1}=b_3$, $a_{i_2}+c_{j_2}=b_3$, $a_{i_3}=c_{j_3}$, and $c_1 \neq c_{j_3}$, where $\{i_1,i_2,i_3\}=\{j_1,j_2,j_3\}=\{1,2,3\}$. By Lemma \ref{l61}, we have 
\begin{center}
$a_{i_1}^2+a_{i_2}^2+a_{i_3}^2=a_1^2+a_2^2+a_3^2=c_1^2+c_2^2+c_3^2=(b_3-a_{i_1})^2+(b_3-a_{i_2})^2+a_{i_3}^2$. 
\end{center}
Simplifying this, we have $2b_3(a_{i_1}+a_{i_2})=2b_3^2$ and hence $a_{i_1}+a_{i_2}=b_3$. Since $a_{i_1}+c_{j_1}=b_3$, this implies that $a_{i_2}=c_{j_1}$. Similarly, since $a_{i_2}+c_{j_2}=b_3$, $a_{i_1}+a_{i_2}=b_3$ implies that $a_{i_1}=c_{j_2}$. However, because $c_1 \neq c_{j_3}$, one of $c_{j_1}$ and $c_{j_2}$ must be $c_1$. This means that either $c_1=a_{i_1}$ or $c_1=a_{i_2}$, which leads to a contradiction because $c_1<a_1\leq a_2 \leq a_3$ by Proposition \ref{p42}. \end{proof}

\section{Conjectural bounds on the dimensions of manifolds with odd numbers of $S^1$ fixed points}\label{s8}

While for any even $k$ and for any $2n$ there is an example of such a manifold $M$, a situation could be different if $k$ is odd. If the number of fixed points is odd, there might be a relation between the dimension of a manifold and the number of fixed points. For the case of three fixed points, actions on $\mathbb{CP}^2$, $\mathbb{HP}^2$, and $\mathbb{OP}^2$ are examples in real dimensions 4, 8, and 16, and these could be the only possible dimensions for compact oriented $S^1$-manifolds with three fixed points. It is plausible that if the number of fixed points is odd, not all dimensions $4n$ are possible. Moreover, the dimension of a manifold might be bounded above by some number that depends on the number of fixed points, as the case of one fixed point supports this idea. Therefore, we suggest the following conjecture.

\begin{conj}\label{c11}
Let the circle act on a compact oriented manifold $M$ with $k$ fixed points, where $k$ is odd. Then $\dim M$ is bounded above by a function of $k$; more precisely, $\dim M \leq 8(k-1)$.
\end{conj}

We specialize in dimension 12. In dimension 12, a standard linear $S^1$-action on $\mathbb{CP}^6$ provides an example with 7 fixed points. Therefore, if $k$ is a positive integer such that $k$ is even or $k \geq 7$, we can construct a 12-dimensional compact oriented $S^1$-manifold with $k$ fixed points from $S^{12}$ and $\mathbb{CP}^6$. On the other hand, there is no known example with 5 fixed points and thus it is a natural question to ask if such a manifold exists.

\begin{que}
Does there exist a circle action on a 12-dimensional compact oriented manifold with exactly 5 fixed points?
\end{que}

We review classification results for circle actions on other types of manifolds and relations between the dimension of a manifold and the number of fixed points. These will provide evidence for Conjecture \ref{c11}.

First, we discuss unitary manifolds. Suppose that the circle group acts on a compact unitary manifold $M$ with $k$ fixed points, $0 < k < \infty$. If $k=1$, $M$ is the fixed point. Suppose $k=2$. Then $S^{2n}$ admits an action with 2 fixed points; hence when $k=2$, examples exist in all even dimensions. Now, assume furthermore that $M$ does not bound a unitary manifold equivariantly. In this case, Kosniowski proved that either $M$ is the 2-sphere or $\dim M=6$ and the weights at the fixed points agree with those of an action on $S^6$ \cite{K2}; also see \cite{K1} for complex manifolds and \cite{PT} for symplectic manifolds. Musin reproved this in \cite{M2}. Based on this result, Kosniowski conjectured that for a compact unitary $S^1$-manifold with $k$ fixed points, $0 < k < \infty$, if $M$ does not bound a unitary manifold equivariantly, then $\dim M<4k$ \cite{K2}. Beyond the case of 2 fixed points, this conjecture remains open in its full generality.

Second, we discuss almost complex manifolds. Consider an $S^1$-action on a compact almost complex manifold $M$ with $k$ fixed points, $0 < k < \infty$. In this case, $M$ does not bound a unitary manifold equivariantly. This is because $M$ bounds a unitary manifold if and only if all the Chern numbers vanish, but the Chern number $c_n[M]$ is equal to the number $k$ of fixed points. Again, if $k=1$, $M$ is the fixed point itself. If $k=2$, by the above either $M$ is the 2-sphere or $\dim M=6$. Moreover, if $k=3$, the dimension of $M$ must be 4 and the weights at the fixed points agree with those of a standard linear action on $\mathbb{CP}^2$ \cite{J2} (also see \cite{J1}). Therefore, for almost complex manifolds, the Kosniowski's conjecture holds for $k \leq 3$, or it holds in dimensions up to 14.

\section*{Funding}

This work was supported by the National Research Foundation of Korea grant funded by the Korea government [2021R1C1C1004158 to D.J.].

\section*{Acknowledgments}

The author would like to thank the anonymous referees for their careful reading of this paper and for providing many fruitful comments, which have helped improve the exposition and quality of this paper.


\begin{thebibliography}{1}

\bibitem[AB]{AB}
M. Atiyah and R. Bott: \emph{The moment map and equivariant cohomology.} Topology \textbf{23} (1984) 1-28.

\bibitem[AS]{AS}
M. Atiyah and I. Singer: \emph{The index of elliptic operators: III.} Ann. Math. \textbf{87} (1968) 546-604.

\bibitem[DW]{DW}
H. Dong and J. Wang: \emph{Non-existence of circle actions on oriented manifolds with three fixed points except in dimensions 4, 8 and 16.} (2023), arXiv:2311.06779.

\bibitem[FS]{FS}
J. Fowler and Z. Su: \emph{Smooth manifolds with prescribed rational cohomology ring.} Geom. Dedicata \textbf{182} (2016), 215-232.

\bibitem[GS]{GS}
L. Godinho and S. Sabatini: \emph{New tools for classifying Hamiltonian circle actions with isolated fixed points.} Found. Comput. Math. \textbf{14} (2014) Issue 4, 791-860.

\bibitem[GW]{GW}
O. Goertsches and M. Wiemeler: \emph{Positively Curved GKM-Manifolds}, Int. Math. Res. Notices \textbf{22} (2015), 12015-12041.

\bibitem[HH]{HH}
H. Herrara and R. Herrera: \emph{$\hat{A}$-genus on non-spin manifolds with $S^1$ actions and the classification of positive quaternion-K\"ahler 12-manifolds.} J. Differ. Geom. \textbf{61} (2002) 341-364.

\bibitem[JT]{JT}
D. Jang and S. Tolman: \emph{Hamiltonian circle actions on eight dimensional manifolds with minimal fixed sets.} Transform. Groups. \textbf{22} (2017), 353-359.

\bibitem[J1]{J1}
D. Jang: \emph{Symplectic periodic flows with exactly three equilibrium points.} Ergod. Theor. Dyn. Syst. \textbf{34} (2014) 1930-1963.

\bibitem[J2]{J2}
D. Jang: \emph{Circle actions on almost complex manifolds with isolated fixed points.} J. Geom. Phys. \textbf{119} (2017) 187-192.

\bibitem[J3]{J3}
D. Jang: \emph{Circle actions on oriented manifolds with discrete fixed point sets and classification in dimension 4.} J. Geom. Phys. \textbf{133} (2018) 181-194.

\bibitem[J4]{J4}
D. Jang: \emph{Circle actions on oriented manifolds with few fixed points}. East Asian Math. J. \textbf{36} (2020) 593-604.


\bibitem[J5]{J6}
D. Jang: \emph{Almost complex torus manifolds - graphs and Hirzebruch genera.}. Int. Math. Res. Notices, Volume 2023, Issue 17, August 2023, Pages 14594-14609.

\bibitem[J6]{J5}
D. Jang: \emph{Circle actions on unitary manifolds with discrete fixed point sets}. Indiana U. Math. J., \textbf{72} No. 6 (2023), 2431-2453.

\bibitem[Kob]{Kob}
S. Kobayashi: \emph{Fixed points of isometries}. Nagoya Math. J. \textbf{13} (1958) 63-68.

\bibitem[Kos1]{K1}
C. Kosniowski: \emph{Holomorphic vector fields with simple isolated zeros.} Math. Ann. \textbf{208} (1974) 171-173.

\bibitem[Kos2]{K2}
C. Kosniowski: \emph{Some formulae and conjectures associated with circle actions}. Topology
Symposium, Siegen 1979 (Proc. Sympos., Univ. Siegen, Siegen, 1979), 331-339, Lecture
Notes in Math., 788, Springer, Berlin, 1980.

\bibitem[Kos3]{K3}
C. Kosniowski: \emph{Fixed points and group actions}. In: Madsen I.H., Oliver R.A. (eds) Algebraic Topology Aarhus 1982. Lecture Notes in Mathematics, vol 1051. Springer, Berlin, Heidelberg.

\bibitem[Ku]{Ku}
A. Kustarev: \emph{Universality of actions on $\mathbb{HP}^2$}. arXiv:1401.4731 (2014).

\bibitem[MS]{MS}
J. Milnor and J. Stasheff: \emph{Characteristic classes.} Princeton University Press, Princeton, N.J.; University of Tokyo Press, Tokyo, 1974. Annals of Mathematics Studies, No. 76.

\bibitem[M]{M2}
O. Musin: \emph{Circle actions with two fixed points}. Math. Notes \textbf{100} (2016) 636-638.

\bibitem[PT]{PT}
A. Pelayo and S. Tolman: \emph{Fixed points of symplectic periodic flows.} Ergod. Theor. Dyn. Syst. \textbf{31} (2011), 1237-1247.

\bibitem[S]{S}
Z. Su: \emph{Rational analogs of projective planes.} Algebr. Geom. Topol. \textbf{14} (2014), 421-438.

\bibitem[T]{T}
S. Tolman: \emph{On a symplectic generalization of Petrie's conjecture.} Trans. Amer. Math. Soc. \textbf{362} (2010), no.8, 3963-3996.

\bibitem[W]{W}
M. Wiemeler: \emph{On circle actions with exactly three fixed points.} (2023), arXiv:2303.15396 

\end{thebibliography}
\end{document}